\documentclass[11pt,thmsa,reqno]{amsart}

\usepackage[all]{xy}
\usepackage{color}
\usepackage{mathtools} 
\usepackage{fancybox}
\usepackage{amssymb}
\usepackage{hyperref}
\usepackage{comment}
\usepackage{graphicx}
\usepackage{faktor}
\usepackage{tikz}
\usepackage{amsthm}
\usepackage{mdframed}

 \newtheorem{theorem}{Theorem}[section]
 \newtheorem*{thm*}{Theorem}

 \newtheorem{proposition}[theorem]{Proposition}

 \newtheorem{definition}[theorem]{Definition}
  \theoremstyle{definition}
 \newtheorem{example}[theorem]{Example}
 \newtheorem{remark}[theorem]{Remark}
  \newtheorem*{remark*}{Remark}

\newcommand{\g}{\ensuremath{\mathfrak{g}}}

\newcommand{\h}{\ensuremath{\mathfrak{h}}}

\newcommand{\rar}[1]{\overset{\rightarrow}{#1}}
\newcommand{\bt}{\mathbf{t}}                  
\newcommand{\bs}{\mathbf{s}} 
\newcommand{\cB}{\mathcal{B}} 

\begin{document}

 \author{Marco Zambon}
\email{marco.zambon@kuleuven.be}
\address{KU Leuven, Department of Mathematics, Celestijnenlaan 200B box 2400, BE-3001 Leuven, Belgium.}


 \title[ ]{Holonomy transformations for Lie subalgebroids}

\begin{abstract}
Given a foliation, there is a well-known notion of holonomy, which can be understood as an action that differentiates to the Bott connection on the normal bundle. We present an analogous notion for Lie subalgebroids, consisting of an effective action of the minimal integration of the Lie subalgebroid,
{and provide an explicit description in terms of conjugation by bisections.
The construction is done in such a way that it easily extends to singular subalgebroids, which provide our main motivation}. 
  \end{abstract}

\date{}
\maketitle

\begin{center}
{\it Dedicated to the memory of Kirill Mackenzie}
\end{center}

\setcounter{tocdepth}{1} 
\tableofcontents

 \section*{{Introduction}}
 
This note extends some well-known geometric constructions from involutive distributions (i.e. foliations) to wide Lie subalgebroids, and even further to singular subalgebroids.
Although at present the author's
 main motivation   is given by  applications for singular subalgebroids, {except for the last section} this note focuses on the case  of wide Lie subalgebroids, for which the constructions can be described more concretely.
Lie algebroids and Lie groupoids were a dear subject to Kirill Mackenzie, who made numerous and important contributions to the subject since its early days. His monograph \cite{MK2} is a standard reference, which is also helpful to make the field accessible to a wide community of researchers. I used Kirill's book in many projects, and the present one is no exception; indeed, this note expands on some of the notions in Chapter 3 of his book. 
 \bigskip

The following is well-known in foliation theory. Let $F$ be an involutive distribution on a manifold $M$; by the Frobenius theorem, it corresponds to a regular foliation, {whose leaves are tangent to $F$}. There is a canonical flat partial connection on the normal bundle $NF{:=TM/F}$ to the foliation, called \emph{Bott connection}. It is given by $$\nabla_X \underline{Y}=\underline{[X,Y]},$$ where $X\in \Gamma(F)$, $Y\in \Gamma(TM)$ and $\underline{Y}:=Y \text{ mod } F$. 

Further, a Lie groupoid integrating the Lie algebroid $F$ is given by  $H(F) \rightrightarrows M$, the \emph{holonomy groupoid} of the foliation.  
{To describe it, for all points $x$ of $M$ fix smoothly a slice $S_x$ transverse to the foliation. Given any two points $x,y$ lying on the same leaf, consider a path $\gamma$ from $x$ to $y$ contained in the leaf.
Extend $\gamma$ smoothly to a family of paths, starting at points of $S_x$ and ending at points of $S_y$, each of which is contained in a leaf. While the extension is not unique, the germ of the resulting diffeomorphism between the slices (mapping starting point to endpoint) is well-defined, and we refer to it as  {\it holonomy transformation}. The holonomy groupoid $H(F)$ consists of all 
paths contained in leaves of $F$, after identifying any two of them if they induce the same holonomy transformation.}
Thus by construction there is an action\footnote{We use the term ``action'' in a loose way, since we only obtain \emph{germs} of diffeomorphisms between slices.}
 of the Lie groupoid $H(F)$ on the fiber bundle $\coprod_{x\in M}S_x\to M$, which moreover is effective  and  
leaves invariant the copy of $M$ given by the ``base points'' of the slices. 
``Effective'' means that the   action map
 $$\chi\colon H(F)\to \coprod_{{x,y}\in M} \textit{GermDiff\,}(S_x,S_y)$$
 is injective. {Further, the choice of slices is immaterial, since if $S_x$ and $S'_x$ are two slices through the point $x\in M$, there is a canonical identification between them  which intertwines holonomy transformations (see for instance \cite[\S2.1, page 22]{FolLieGr}).
 }

The relation between the above objects is the following:

\begin{center} 
\emph{The action $\chi$ of the holonomy groupoid differentiates   to the Bott connection $\nabla$.} 
\end{center}
We describe in detail this relation, following \cite[Lemma 3.11]{AZ2}. We have:
\begin{enumerate}
\item  [i)] A Lie groupoid action $\chi$ of $H(F)$ on $\coprod_{x\in M}S_x$ (holonomy action),
\item [ii)] A Lie groupoid representation   of $H(F)$ on $NF$ (linearized holonomy),
\item [iii)] A Lie algebroid representation $\nabla$ of $F$ on $NF$ (Bott connection).
\end{enumerate}
For the passage i) $\to$ ii), given the action map $\chi$, 
 take the derivative of the holonomy transformations at the base points of the slices.
 For the passage ii) $\to$ iii),  view the representation of $H(F)$ as a Lie groupoid morphism  $\Psi\colon H(F)\to \textrm{Iso}(NF)$, and  
take  the associated Lie algebroid morphism.

\bigskip
 \paragraph{\bf Main results.}
Now let $A$ be a Lie algebroid over a manifold $M$ and $B$ a wide Lie subalgebroid \cite[Def. 3.3.21]{MK2}. The same formula as for the Bott connection defines a representation $\nabla$ of $B$ on $A/B$.
Assume that $A$ is an integrable Lie algebroid, and fix a source connected Lie groupoid $G\rightrightarrows M$ integrating it. Moerdijk-Mr{\v{c}}un \cite{MMRC} showed that there is a ``minimal'' Lie groupoid $H_{min}\rightrightarrows M$ integrating $B$ endowed with an immersion $\Phi\colon H_{min}\to G$. Choose smoothly, for every $x\in M$, a slice $S_x\subset \bs^{-1}(x)$   transverse to ${B}_x$  inside the source fiber $\bs^{-1}(x)\subset G$. Here ${B}_x$  denotes the fiber of $B$ at $x$. 

 We show  in  Thm.~\ref{thm:main} and Prop.~\ref{prop:Bott}:
\begin{itemize}
\item[i)]
There is a canonical action of $H_{min}$ on the fiber bundle $\coprod_{x\in M}S_x$, given by a groupoid morphism  
$$\chi\colon H_{min}\to \coprod_{{x,y}\in M} \textit{GermDiff\,}(S_x,S_y).$$
\item[ii)] {The above action $\chi$ gives rise to a representation of $H_{min}$ on the vector bundle $A/B$ over $M$,  by taking derivatives of germs of diffeomorphisms.}
\item[iii)] 
{The above representation of $H_{min}$ differentiates}
to the canonical representation $\nabla$ of $B$ on $A/B$ (the ``Bott connection'').  \end{itemize}
 The morphism $\chi$ in i) is easily obtained following the constructions given in  {\cite[\S2]{MMRC}, and up to isomorphism it is independent of the choice of slices}.
One of the contributions of this note is to provide an alternative characterization in terms of conjugations by bisections (see Prop.~\ref{prop:agree}).
We show that while $\chi$ is not injective, the pair $(\chi,\Phi)$ does define an injective map 
on $H_{min}$ (see Remark ~\ref{rem:inj}). For the sake of exposition, we spell out the above results for the case of Lie subalgebras in Ex.~\ref{ex:lieal},~\ref{ex:liealdescr},~\ref{ex:liealBott}.  Finally,  in Thm.~\ref{thm:mainsing} 
we extend item i) above from  wide Lie subalgebroids to  the {singular subalgebroids} {introduced in \cite{AZ3}}, {and comment on how the other results extend to singular subalgebroids.}

 \bigskip
\noindent{\bf Motivation.}
Our motivation to consider $\chi$ comes from the integration of singular subalgebroids, which we address in  \cite{AZ4}. 
A singular subalgebroid is a submodule $\cB$ of the compactly supported sections of $A$ which is locally finitely generated and involutive. Canonically associated to it  there is a topological groupoid  $H^G(\cB)\rightrightarrows M$, which comes with a morphism $\Phi\colon H^G(\cB)\to G$ \cite{AZ3}.
In Thm.~\ref{thm:mainsing} we extend item 1) above showing that $H^G(\cB)$
admits an   {effective} action on the union of slices. This means to each element of $H^G(\cB)$ there is an associated holonomy transformation, in an \emph{injective} way.
This property -- which is satisfied by  $H^G(\cB)$ but clearly fails for any non-trivial quotient of it --  singles   out $H^G(\cB)$ as a special integration.

Another motivation is the following: the map $\chi$ will be relevant when extending some of the results of \cite{GarmendiaVillatoro} from singular foliations to singular subalgebroids. 
Recall that there Garmendia-Villatoro provide an alternative construction of the holonomy groupoid of a singular foliation $\mathcal{F}$, by taking a quotient of the space of $\mathcal{F}$-paths
(this is closer in spirit to the integration of Lie algebroids by Lie algebroid paths).

 \bigskip
\noindent{\bf Relation to homogeneous spaces.} 
{ 
Let again $A$ be a Lie algebroid over $M$, $B$ a wide Lie subalgebroid, and $G\rightrightarrows M$ a source-connected Lie groupoid integrating $A$. }

{
Assume that $B$ is integrated by a \emph{closed, embedded} Lie subgroupoid $H$ of $G$.
 Then the fiber bundle $\coprod_{x\in M}S_x$, and the $H$-action of i) above, admit a simpler description. Consider the orbit space $H\backslash G$ of the $H$-action by left multiplication on the map $\bt\colon G\to M$, together with the map $H\backslash G\to M$ induced by the source map of $G$. The above fiber bundle is then canonically isomorphic to a neighborhood of the identity section of $H\backslash G$. There is a (left) action of $H$ on $H\backslash G$, used in \cite{LGVoglaire},
for which an element $h$ acts by $[g]\mapsto [g\cdot h^{-1}]$. Under the canonical identification, this action agrees with the action of i), as one sees using eq. \eqref{eq:formualconj}. This observation simplifies the arguments in our Example \ref{rem:normal}.
Further this observation implies that, in the setting of closed embedded Lie subgroupoids, the results of Prop.~\ref{prop:Bott} already appeared  in \cite{LGVoglaire}, although somewhat implicitly. (See Remark 4.3 and the proof of Prop. 4.14 in \cite{LGVoglaire}, where the arguments are similar to those we give in \S\ref{subsec:firstproof}.) The main objective of  \cite{LGVoglaire} is to relate the linearizability of the $H$-action and the Atiyah class \cite{AtHomLeibniz}.
} 

{In the general case, i.e. when the  immersion $\Phi\colon H_{min}\to G$ is not an embedding, one can work with local Lie groupoids. Denote by $G_{loc}$ and $H_{loc}$ small enough neighbourhoods of the identity sections in $G\rightrightarrows M$ and $H_{min}\rightrightarrows M$ respectively. Then $H_{loc}\backslash G_{loc}$  is canonically isomorphic to $\coprod_{x\in M}S_x$, and carries a (left) action of $H_{loc}$ given by $[g]\mapsto [g\cdot \Phi(h^{-1})]$ as above. This action extends to an $H_{min}$-action, but not simply using the same formula: for $h\in H_{min}$, in general $g\cdot \Phi(h^{-1})$ lies ``far'' from the identity section of $G$ and thus is not an element of $G_{loc}$. Instead, the  $H_{min}$-action can be described as conjugation by a suitable bisection: under the action of $h\in H_{min}$, the class
$[g]$ is mapped to $[\alpha(\bt(g))\cdot g \cdot\Phi(h)^{-1}]$,
where  
$\alpha:=\Phi(\alpha_H)$ for some bisection  $\alpha_H$ of $H_{min}$ through $h$.
This follows from our considerations in \S\ref{sec:conj}, and we were not able to locate it elsewhere in the literature.
 }
 
 {If we replace the wide Lie subalgebroid $B$ by a singular subalgebroid $\cB$,
  one is forced to work with the collection of slices $\coprod_{x\in M}S_x$, since the description  as an orbit space no longer holds.  
   This is the reason why, in this note, we   work exclusively with slices.}
  
 \bigskip
\noindent{\bf Conventions.} 
The
Lie algebroid of the Lie groupoid over $M$ is given by $\ker(\bs_*)|_M$, where $\bs$ denotes the source map. 
Bisections of the Lie groupoid are by default locally defined, and will be understood both as submanifolds (transverse to the source and target fibers), and as sections of the source map (inducing a diffeomorphism on the base manifold $M$). 
 
 \bigskip
  \paragraph{\bf Acknowledgements.} We thank Karandeep Singh and Joel Villatoro for discussions related to this note. {We also thank the two referees for their valuable comments, in particular for pointing out the relation of the holonomy action  to homogeneous spaces.}
We acknowledge partial support by the long term structural funding -- Methusalem grant of the Flemish Government, the FWO and FNRS under EOS project G0H4518N, the FWO research project G083118N (Belgium). 

\section{The holonomy map $\chi$}

Let $A$ be an integrable Lie algebroid over a manifold $M$, fix a source connected Lie groupoid $G\rightrightarrows M$ integrating $A$.
Let $B$ be a wide\footnote{This means that $B$ is a subbundle over the whole of $M$ \cite[Def. 3.3.21]{MK2}. In the following we will often write ``Lie subalgebroid'' to mean ``wide Lie subalgebroid''. In the terminology of \cite{AtHomLeibniz}, $(A,B)$ is called Lie pair.
} Lie subalgebroid of $A$; {we refer the reader to \cite[\S1]{LGVoglaire} for a list of examples.}
 Denote by $H^G(B)\rightrightarrows M$ the minimal integral\footnote{{In the introduction, the minimal integral was denoted by $H_{min}$. Here we adopt the notation $H^G(B)$,
since it is the holonomy groupoid of Lie subalgebroid $B$ w.r.t. Lie groupoid $G$ \cite{AZ3}
}} 
 of $B$ over $G$ \cite[\S2]{MMRC}, i.e. the smallest Lie groupoid integrating $B$ that admits a  Lie groupoid morphism $$\Phi\colon H^G(B)\to G$$   integrating the inclusion $B\to A$. The morphism $\Phi$ is unique, it is   an immersion, but not necessarily injective. The source and target maps of $G$ will be denoted by $\bs,\bt$, those of $H^G(B)$ by  $\bs_H,\bt_H$. 

Choose smoothly, for every $x\in M$, a slice $S_x$ through $x$ inside the source fiber $\bs^{-1}(x)$ of $G$ so that $T_xS_x$ is transverse to ${B}_x$ (both are subspaces of $\ker(\bs_*)_x)$.

In this section we construct the holonomy map $\chi$, and we give an alternative characterization in terms of bisections.

\subsection{Existence of the holonomy map $\chi$} 
 
\begin{theorem}\label{thm:main}
There is a canonical  groupoid morphism  
$$\chi\colon H^G(B)\to \coprod_{{x,y}\in M} \textit{GermDiff\,}(S_x,S_y)$$
 mapping $h\in  H^G(B)$ to a germ of diffeomorphism
$$S_{\bs_H(h)}\to S_{ \bt_H(h)}.$$
\end{theorem}

\begin{proof}

To describe $\chi$, we first
recall the construction of the minimal integral from \cite[\S2]{MMRC}.
From $B$ we obtain an involutive distribution $\rar{B}$ on $G$, given by 
$\rar{B}_g:=(R_g)_* B_{\bt(g)}$, where $R_g$ denotes the right-translation by $g\in G$. Notice that this distribution is tangent to the source fibers of $G$ and right-invariant. We denote  the holonomy groupoid of this foliation by $$H(\overset{\rightarrow}{B})\rightrightarrows G.$$
The Lie groupoid $G$ acts on itself by right-translation. Hence it also acts on the holonomy groupoid $H(\overset{\rightarrow}{B})$, by right-translating equivalence classes of paths  in the leaves of $\overset{\rightarrow}{B}$. This provides a principal right  action of $G$ on the Lie groupoid $H(\overset{\rightarrow}{B})$, and the quotient Lie groupoid is precisely the minimal integral $H^G(B)\rightrightarrows M$. Notice that the quotient map $\pi$ covers $\bt\colon G\to M$:
\begin{equation}\label{eqn:quot}
 \xymatrix{
H(\overset{\rightarrow}{B})  \ar[d] \ar@<-1ex>[d]  \ar[r]^{\pi} &H^{G}(\cB)  \ar[d] \ar@<-1ex>[d]  \\
G \ar[r]^{\bt} &   M }
\end{equation}
The morphism $\Phi\colon H^G(B)\to G$ is obtained  quotienting  the natural map from $
H(\overset{\rightarrow}{B})$ to the holonomy groupoid of the foliation by $\bs$-fibers, which is   $G   \times _{(\bs,\bs)} G$.

{\it Step 1.}
By definition of holonomy groupoid of a foliation,  there is an injective canonical groupoid morphism from 
$H(\overset{\rightarrow}{B})$ to the germs of diffeomorphisms between slices transverse to $\rar{B}$ in $G$. Through every point $g$ of $G$ we have a slice transverse to $\rar{B}_g$ inside the corresponding source fiber, by setting $S_g:=(R_g)_*S_{\bt(g)}$. By restriction to the  slices in these source-fibers, we obtain a   groupoid morphism\footnote{This morphism is trivial -- i.e. the foliation $\overset{\rightarrow}{B}$ has trivial holonomy -- if{f} the map $\Phi\colon H^G(B)\to G$ is injective. This follows from  \cite[Cor. 2.5]{MMRC}.}
$$\Xi\colon H(\overset{\rightarrow}{B}) \to  \coprod_{{g,g'}\in G} \textit{GermDiff\,}(S_g,S_{g'}),$$
which is injective, as a consequence of the fact that the regular foliation on $G$  by  $\bs$-fibers  has trivial holonomy.

We make an observation about $\Xi$. 
Let $\xi\in H(\overset{\rightarrow}{B})$ with source $x\in 1_M\subset G$, and let $g\in G$ with $\bt(g)=x$.
The germs of diffeomorphisms associated to $\xi$ and $\xi \cdot g$ (the image of $\xi$ under the action of the element $g$) are related by 
\begin{equation}\label{eq:rightinv}
\Xi(\xi)=R_g^{-1}\circ \Xi(\xi\cdot g)\circ R_g.
\end{equation}
This follows from the facts that  the involutive distribution $\overset{\rightarrow}{B}$ is right-invariant, and the action of $G$ on the holonomy groupoid $H(\overset{\rightarrow}{B})$ is induced by right-translation.

{\it Step 2.} 
For every $h\in H^G(B)$, there is a 
distinguished preimage $\xi\in H(\overset{\rightarrow}{B})$ under $\pi$, namely
the unique preimage with source equal to $1_{\bs_H(h)}$. 
The target of $\xi$ is then $\Phi(h)$, hence
the corresponding germ of diffeomorphism reads  $\Xi(\xi)\colon S_{1_{\bs_H(h)}}\to S_{ \Phi(h)}$.
We define $\chi$ by 
\begin{equation}\label{eq:defchi}
\chi(h):=R_{\Phi(h)}^{-1}\circ \Xi(\xi).  
\end{equation}

{\it Step 3.}
The map $\chi$ we just defined is a groupoid morphism, by the right-invariance  property expressed in Eq.~\eqref{eq:rightinv}. In more detail: given composable $h_2,h_1\in H^G(B)$, we have
$$
\chi(h_2)\circ\chi(h_1)=R_{\Phi(h_2)}^{-1}\circ \Xi(\xi_2)\circ R_{\Phi(h_1)}^{-1}\circ \Xi(\xi_1)
=R_{\Phi(h_2h_1)}^{-1}\circ \Xi(\xi_2\cdot \Phi(h_1))\circ \Xi(\xi_1)=\chi(h_2h_1).
$$
Here in the second equality we used Eq. ~\eqref{eq:rightinv} with  $g=\Phi(h_1)$ to rewrite $\Xi(\xi_2)$, and in the third equality  that $(\xi_2\cdot \Phi(h_1))\xi_1 \in H(\overset{\rightarrow}{B})$ is the distinguished  preimage of $h_2h_1$ under $\pi$. 
\end{proof}

\begin{figure}
 \begin{tikzpicture}[scale=1.0]
 \draw  (-3,0) -- (0,3) -- (3,0) -- (0,-3) -- cycle;
  \draw (-3,0)  -- (3,0);
\node [below] at (-1.5,0) {$M$};
\node [above] at (0,-2) {$G$};
  
\def\ra{0.2}
 \draw[thick,blue] (0.5+\ra,-0.5-\ra)  -- (0.5-\ra,-0.5+\ra);
\draw[thick,blue] (0+\ra,0-\ra)  -- (0-\ra,0+\ra);
\node [] at (-0,0) {$\bullet$};
\draw[thick,blue] (-0.5+\ra,0.5-\ra)  -- (-0.5-\ra,0.5+\ra);
\draw[thick,blue] (-1+\ra,1-\ra)  -- (-1-\ra,1+\ra);
\node [] at (-1,1) {$\bullet$};

 \draw [->,red] (0.1,0.1) to [out=105,in=-15] (-1+0.1,1.1);
  \node[red]  at (0,1) {$\Xi(\xi)$}; 
 
\node [] at (1,1) {$\bullet$};
\node [right] at (1,1) {$\;g$};
\draw[thick,blue] (1.5+\ra,0.5-\ra)  -- (1.5-\ra,0.5+\ra);
\draw[thick,blue] (1+\ra,1-\ra)  -- (1-\ra,1+\ra);
\draw[thick,blue] (0.5+\ra,1.5-\ra)  -- (0.5-\ra,1.5+\ra);
\draw[thick,blue] (0+\ra,2-\ra)  -- (0-\ra,2+\ra);
\node [] at (0,2) {$\bullet$};

 \draw [->,red] (1.1,1.1) to [out=105,in=-15] (0.1,2.1);
 \node[red]  at (1.5,2) {$\Xi(\xi\cdot g)$}; 
\end{tikzpicture}

\caption{The Lie groupoid $G$,
 with the right-invariant distribution $\rar{B}$ (in blue).
 An element $\xi\in H(\overset{\rightarrow}{B})$ induces a diffeomorphism $\Xi(\xi)$ between slices (slices not depicted). It is related to  $\Xi(\xi\cdot g)$ as explained in Step 1 of Thm.~\ref{thm:main}.
 }
\end{figure}
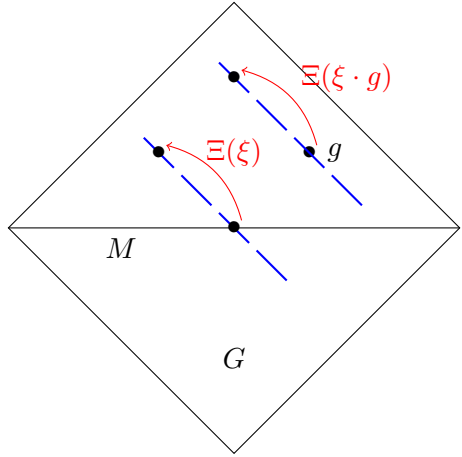

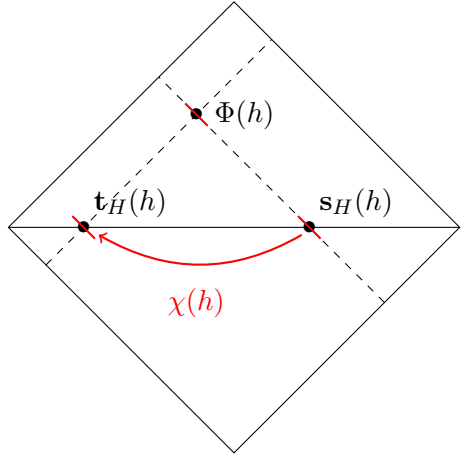
\begin{figure}
 \begin{tikzpicture}[scale=1.0]
 \draw  (-3,0) -- (0,3) -- (3,0) -- (0,-3) -- cycle;
  \draw (-3,0)  -- (3,0);
    \draw[dashed] (-2.5,-0.5)  -- (0.5,2.5); 
  \draw[dashed] (2,-1)  -- (-1,2); 

 \def\ra{0.15}
  
\node [] at (1,0) {$\bullet$};
\node [ above right]  at (1,0) {$\bs_H(h)$};
\draw[thick,red] (1+\ra,0-\ra)  -- (1-\ra,0+\ra);

\node [] at (-0.5,1.5) {$\bullet$};
\node [  right] at (-0.5,1.5) {$\;\Phi(h)$};
\draw[thick,red] (-0.5+\ra,1.5-\ra)  -- (-0.5-\ra,1.5+\ra);

\node [] at (-2,0) {$\bullet$};
\node [ above right] at (-2,0) {$\bt_H(h)$};
\draw[thick,red] (-2+\ra,0-\ra)  -- (-2-\ra,0+\ra);

\draw [->,thick,   red] (1-0.1,0-0.1) to [out=210,in=-30] (-2+0.2,0-0.1);
\node [red] at (-0.5,-1) {$\chi(h)$}; 	
\end{tikzpicture}
\caption{Given $h\in H^G(B)$, the construction of the map $\chi(h)$ as in Step 2 of Thm.~\ref{thm:main}.
In this picture the slices $S$ are zero-dimensional, but  are depicted as short red segments.}
\end{figure}

\begin{remark}
In the two extreme cases in which $B$ has rank zero or $B=A$,
the  map $\chi$ of Thm.~\ref{thm:main} is uninteresting: in the former case $H^G(B)$ is the trivial groupoid $M \rightrightarrows M$ and $\chi(x)=Id_{S_x}$ for all $x\in M$. In the latter case, the slices $S_x$ are all zero-dimensional.
\end{remark}

We check that, when $B$ is an involutive distribution, the morphism $\chi$ of Thm.~\ref{thm:main} recovers the usual holonomy map for foliations.
\begin{example}[Foliations]\label{ex:fol}
Let $F$ be an involutive distribution on $M$, i.e. a Lie subalgebroid of $A=TM$. Fix slices $S_x$ transverse to $F$ at all $x\in M$. We show that the map $\chi$ from Thm.~\ref{thm:main} is exactly the usual holonomy action of the foliation tangent to $F$.
Indeed, the pair groupoid $M\times M$ integrates $TM$, and $H^{M\times M}(F)=:H(F)\rightrightarrows M$ is the usual holonomy groupoid of {the foliation tangent to} $F$.
The resulting involutive distribution on the pair groupoid is $\overset{\rightarrow}{F}=F\times \{0\}$, the product  of $F$ and the trivial distribution on $M$. Hence 
its holonomy groupoid is {the product groupoid} $H(\overset{\rightarrow}{F})=H(F)\times M$.

Let $h\in H(F)$, denote $x=\bs_H(h)$ and $y=\bt_H(h)$, and denote by $hol_h\colon S_{x}\to S_{y}$ the associated germ of diffeomorphism. Consider the element $(h,1_x)$ of $H(\overset{\rightarrow}{F})$. Under the identification $$\bs^{-1}(x)=M\times \{x\}\cong M,\; (z,x)\mapsto z$$ given by the target map, the restriction of $\overset{\rightarrow}{F}$ to $\bs^{-1}(x)$ is just $F$. Hence
the holonomy diffeomorphism $\Xi(h,1_x)\colon S_{(x,x)}\to S_{(y,x)}$ between  slices contained in $\bs^{-1}(x)$, is exactly $hol_h$, upon the identification $S_{(y,x)}\cong S_y$.
Thus $\chi(h)=R_{(y,x)}^{-1}\circ \Xi(h,1_x)$ is precisely $hol_h\colon S_{x}\to S_{y}$.
\end{example}

\begin{remark}[Injectivity of $\chi$]\label{rem:inj}
The   morphism $\chi$ of Thm.~\ref{thm:main} is not injective in general. It is injective  when $A=TM$ (i.e. for foliations, by Ex.~\ref{ex:fol}) but not when $A$ is a Lie algebra, see Ex.~\ref{ex:lieal} below.
However the groupoid morphism  
$$(\chi,\Phi)\colon H^G(B)\to \coprod_{{x,y}\in M} \textit{GermDiff\,}(S_x,S_y)\,\times\, G$$
is injective. Indeed,  the composition {of this map} with right translation gives the injective assignment  $$h\mapsto (\chi(h),\Phi(h))\mapsto R_{\Phi(h)}\circ \chi(h)=\Xi(\xi).$$ Here $\Xi$ is the morphism appearing in the proof of Thm.~\ref{thm:main}, which is injective, and $\xi$ is the distinguished element of $H(\overset{\rightarrow}{B})$ associated to $h$. Notice that since both $\chi$ and $\Phi$ are morphisms  with base map $Id_M$, the morphism $(\chi,\Phi)$ takes values in the subgroupoid consisting of pairs of elements with matching source and matching target, which is a groupoid over $M$.
\end{remark}

\begin{remark}[Triviality of $\chi$]\label{rem:chitrivial}
One can wonder when the morphism   $\chi$ is a trivial groupoid morphism, i.e. when it maps all elements in the isotropy groups of  $H^G(B)$ to (germs of) identity diffeomorphisms on slices.

When $A$ is a Lie algebra (thus $B$ is a Lie subalgebra),   $H^G(B):=H$ is a connected Lie subgroup of $G$. In this case, $\chi$ is trivial if{f}
$H$ is a normal subgroup of $G$,
see Remark~\ref{rem:normal} below.

When $B$ is tangent distribution of a foliation, we just saw in Remark~\ref{rem:inj} that $\chi$ is injective,
so $\chi$ is trivial if{f} the isotropy groups of the holonomy groupoid of the foliation are trivial (i.e. the foliation has trivial holonomy).

\end{remark}

\subsection{The construction of $\chi$}\label{subsec:constr}
We make the morphism $\chi$ of Thm.~\ref{thm:main} more concrete, by making $H(\overset{\rightarrow}{B})\rightrightarrows G$ more concrete and by using  bisections to   describe the holonomy of a foliation.  

A theorem from \cite[\S 2.1]{AZ4}, specialized to the Lie subalgebroid $B$, reads:

\begin{theorem}\label{thm:trafogroidreg}
\label{thm:MoeMrcreg}
i) The Lie groupoid  $H(\overset{\rightarrow}{B})$ is canonically isomorphic 
to the transformation groupoid $$H^{G}({B})\times_{\bs_H,\bt}G.$$ 
The latter is the transformation groupoid of the action of $H^{G}(B)$ on the map $\bt \colon G \to M$ given by  $(h,g)\mapsto \Phi(h)g$.

ii) Under this isomorphism, 
the upper horizontal map $\pi$ in diagram~\eqref{eqn:quot} is  $(h,g)\mapsto h$.  
\end{theorem}

 \begin{remark}\label{rem:holbi}
Consider a regular foliation on a manifold $M$, {tangent to an involutive distribution $F$}, and a choice of slice $S_x$ at every point $x$ of $M$. Given an element $h$ of the holonomy groupoid $H(F)\rightrightarrows M$, the corresponding holonomy transformation $S_{\bs(h)}\to S_{\bt(h)}$ can be described as follows: take any   bisection of $H(F)$ through $h$ so that the diffeomorphism it induces on (open subsets of) $M$ maps $S_{\bs(h)}$ to $S_{\bt(h)}$, and restrict it to $S_{\bs(h)}$. Notice that while such a bisection is not unique, the germ of its restriction to $S_{\bs(h)}$ is unique. Indeed, there is a unique germ of section of $\bs|_{S_{\bs(h)}}$ through $h$ which maps to $S_{\bt(h)}$ under the target map, as a consequence of the fact that the isotropy groups of $H(F)$ are discrete. \end{remark}

Thanks to the two ingredients above, we can  give the following description of the morphism $\chi$ 
without mentioning explicitly $H(\rar{B})$. (To check this, follow the proof of Thm.~\ref{thm:main}.) 

\bigskip

\begin{mdframed}
\begin{itemize}
\item Take $h\in H^G(B)$. 
\item Consider $(h,1_x)\in H^{G}({B})\times_{\bs_H,\bt}G$, where $x=\bs_H(h)$. 
\item Any bisection
  through $(h,1_x)$ is of the form $\{(\sigma(g),g): g\in G\}$ for a map $$\sigma \colon G\to H^G(B)$$ with $\bs_H\circ \sigma=\bt$ and $\sigma(1_x)=h$ (defined near $1_x$). 
 The induced diffeomorphism of (open subsets of) $G$ is $g\mapsto \Phi(\sigma(g))\cdot g$. Now choose $\sigma$ so that the induced diffeomorphism maps  $S_{1_{x}}$  to $S_{\Phi(h)}$. Notice that its restriction to $S_{1_{x}}$ is precisely
$\Xi (h,1_x)\colon  S_{1_{x}}\to S_{\Phi(h)}$,  by construction and by Remark~\ref{rem:holbi}.
\item Hence  $\chi(h)$ reads
 $$\chi(h)\colon S_{\bs_H(h)}\to S_{ \bt_H(h)}, \;\;g\mapsto \Phi(\sigma(g))\cdot g \cdot\Phi(h)^{-1}.$$
  \end{itemize}
\end{mdframed}

\smallskip
\begin{example}[Lie subalgebras]\label{ex:lieal}
Let  $\g$ be a Lie algebra, $G$ an integrating connected Lie group, and $\h$ a Lie subalgebra. Then $H^G(\h)=:H$ is the connected Lie subgroup of $G$ with Lie algebra $\h$, and $\Phi\colon H \to G$ the inclusion.
Since $M$ is a point, namely the unit $e$, we only need to fix a slice $S_e\subset G$ through $e$ transverse to $H$.

We now describe $\chi$ following the above steps.
Take $h\in H$, and  
consider the element $(h,e)$ of the transformation groupoid $H\times G\rightrightarrows G$, {where the action of $H$ on $G$ is by left multiplication}. A bisection through this element is constructed out of a locally defined map $\sigma\colon G\to H$ with $\sigma(e)=h$; choose $\sigma$ so that the induced diffeomorphism $g\mapsto \sigma(g)\cdot g$ between neighborhoods (of $e$ and $h$ respectively) in $G$ has the property that it maps $S_e$ to $R_h(S_e)$. Then we have
$$\chi(h)\colon S_{e}\to S_{e}, \;g\mapsto  \sigma(g) \cdot g \cdot h ^{-1}.$$

As $\sigma$ it is tempting to choose the map $g\mapsto g\cdot h \cdot g^{-1}$, because then $g\mapsto \sigma(g)\cdot g=gh$ does send $S_e$ to $R_h(S_e)$. However this candidate for $\sigma$ does not take values in $H$ in general. It does when $H$ is a normal subgroup, and in that case we obtain $\chi(h)=Id_{S_e}$ for all $h$. In Remark~\ref{rem:normal} we will  refine this observation with a slightly different proof.
 \end{example}

\section{A description of $\chi$ via conjugation}\label{sec:conj}

As in the previous section, let $A$ be an integrable Lie algebroid over a manifold $M$,    $G\rightrightarrows M$ a Lie groupoid integrating $A$, and $B$ be a wide Lie subalgebroid of $A$. Recall that we have a canonical Lie groupoid morphism $\Phi\colon H^G(B)\to G$. 
Choose smoothly, for every $x\in M$, a slice $S_x$ inside the source fiber $\bs^{-1}(x)$ 
transverse to $B_x$.

Here, for all $h\in H^G(B)$, we construct a map $\chi^{conj}	(h)$ between slices, and in Prop.~\ref{prop:agree}  we compare it with  the map $\chi(h)\colon S_{\bs_H(h)}\to S_{ \bt_H(h)}$ defined in Thm.~\ref{thm:main}. The definition of $\chi^{conj}(h)$ is as follows.

\bigskip
\begin{mdframed}
\begin{itemize}
\item Take $h\in H^G(B)$. 
\item Take a bisection $\alpha_H$ of $H^G(B)$ through $h$
\item Denote $\alpha:=\Phi(\alpha_H)$, a bisection of $G$ through $\Phi(h)$.
Consider $L_{\alpha}$,  the left multiplication by $\alpha$ restricted to 
$S_{\bs_H(h)}$.
\item Define  $\chi^{conj}(h):= R_{\Phi(h)^{-1}}\circ L_{\alpha}$.
Since $S_{\bs_H(h)}$ is contained in a source fiber,  this map is just conjugation by the bisection: 
$$\chi^{conj}(h)=R_{\alpha^{-1}}\circ L_{\alpha},$$
 where $\alpha^{-1}$ is the bisection inverse to $\alpha$.
   \end{itemize}
 \end{mdframed}
 
 \bigskip
Spelled out, the map reads
 \begin{align}\label{eq:formualconj}
\chi^{conj}(h)\colon S_{\bs_H(h)}&\to R_{\Phi(h)^{-1}} L_{\alpha}(S_{\bs_H(h)})\\
\nonumber g &\mapsto  (\Phi\circ \alpha_H\circ\bt)(g)\cdot g \cdot\Phi(h)^{-1}. 
\end{align}

\begin{figure}[h]
 \begin{tikzpicture}[scale=1.0]
 \draw  (-3,0) -- (0,3) -- (3,0) -- (0,-3) -- cycle;
  \draw (-3,0)  -- (3,0);
    \draw[dashed] (-2.5,-0.5)  -- (0.5,2.5); 
  \draw[dashed] (2,-1)  -- (-1,2); 

 \def\ra{0.15}
  
\node [] at (1,0) {$\bullet$};
\node [ above right]  at (1,0) {$\bs_H(h)$};
\draw[thick,red] (1+\ra,0-\ra)  -- (1-\ra,0+\ra);

\node [] at (-0.5,1.5) {$\bullet$};
\node [ above ] at (-0.5,1.7) {$\;\Phi(h)$};

\draw [blue, thick] (-0.5,1.5) to [out=0,in=120] (-0.5+0.4,1.5-0.1);
\draw [blue,thick] (-0.5,1.5) to [out=180,in=300] (-0.5-0.4,1.5+0.1);
\node [below, right] at (-0.5+0.4,1.5-0.1) {$\;\alpha$};
         
\node [] at (-0.5,-1.5) {$\bullet$};
\node [ right] at (-0.5,-1.5) {$\Phi(h)^{-1}$};

\node [] at (-2,0) {$\bullet$};
\node [ above right] at (-2,0) {$\bt_H(h)$};
\draw[thick,red] (-2+\ra,0-\ra)  -- (-2-\ra,0+\ra);

\draw [->,thick,   red] (1-0.1,0-0.1) to [out=210,in=-30] (-2+0.2,0-0.1);
\node [red] at (-0.5,-0.8) {$\chi^{conj}(h)$}; 	
\end{tikzpicture}
\caption{Given $h\in H^G(B)$, the construction of the map $\chi^{conj}(h)$.
The slices $S$ are zero-dimensional, but  are depicted as short red segments.}
\end{figure}
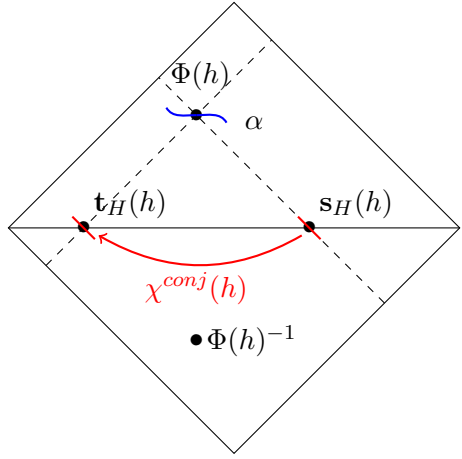

 \begin{remark}
  To construct the map $\chi^{conj}(h)$, we  
choose a bisection $\alpha_H$ in $H^G(B)\rightrightarrows M$, rather than one in $H(\overset{\rightarrow}{B})\rightrightarrows  G$ as we did in \S\ref{subsec:constr}. 
The map -- in particular its codomain -- depends on the bisection.
However  this dependence is immaterial; see Prop.~\ref{prop:agree} below and its proof.
\end{remark}

In the third item above, in general there is no bisection $\alpha_H$ so that  
$L_{\alpha}(S_{\bs_H(h)})$ equals  $S_{\Phi(h)}:=R_{\Phi(h)}(S_{\bt_H(h)})$, due to the fact that  the former slice is obtained by left-multiplication and   the latter slice by  right-multiplication. (Notice that both slices are transverse to $\overset{\rightarrow}{B}_{\Phi(h)}$, as a consequence of the fact that $\alpha$ is the image under $\Phi$ of a bisection of the Lie groupoid $H^G(B)$.)
Consequently, the same holds
 for the slices $R_{\Phi(h)^{-1}} L_{\alpha}(S_{\bs_H(h)})$ and $S_{\bt_H(h)}$.  
However, since both slices pass through the same point $1_{\bt_H(h)}$ and are transverse to the distribution $\overset{\rightarrow}{B}$, there is a canonical identification $\Psi_{\bt_H(h)}$ between $R_{\Phi(h)^{-1}} L_{\alpha}(S_{\bs_H(h)})$ and $S_{\bt_H(h)}$, obtained ``sliding'' along the leaves of the distribution.

\begin{proposition}\label{prop:agree}
For all $h\in H^G(B)$, the maps $\chi(h)$ and $\chi^{conj}(h)$ agree 
under the canonical identification  $\Psi_{\bt_H(h)}$ between their codomains.
\end{proposition}

{As a consequence of this proposition and eq. \eqref{eq:formualconj}, the action $\chi$ of $H^G(B)$ on the fiber bundle $\coprod_{x\in M}S_x$ is equivalent to its action on the orbit space $H_{loc}\backslash G_{loc}$ described in the introduction.}

  \begin{proof}
Take $h\in H^G(B)$, and a bisection $\alpha_H$ of $H^G(B)$ through $h$. We will express $\chi(h)$ using these data, following the description of $\chi$ given in 
the proof of Thm.~\ref{thm:main}, and using  several arguments encountered in 
\S~\ref{subsec:constr}.

 Take the preimage of 
 the bisection $\alpha_H$ under the map $\pi\colon H^{G}({B})\times_{\bs_H,\bt}G\to H^{G}(B)$ (see Thm.~\ref{thm:MoeMrcreg}). We obtain the bisection
 $\{(\alpha_H(\bt(g)),g):g\in G\}$ through $(h,1_{\bs_H(h)})$. 
 The diffeomorphism carried by this bisection is 
\begin{equation*}
g\mapsto \Phi((\alpha_H\circ\bt)(g)))\cdot g,
\end{equation*}
which is just the left translation  by the bisection 
$\alpha:=\Phi\circ \alpha_H$ of $G$.

Denote by $L_{\alpha}$
the   restriction  to $S_{1_{\bs_H(h)}}$ of the left translation  by $\alpha$; it does not map to 
$S_{\Phi(h)}$ in general. However $L_{\alpha}$
 maps to 
a slice through $\Phi(h)$ which is canonically identified with
$S_{\Phi(h)}$, and under this identification it agrees with the holonomy diffeomorphism 
$\Xi (h,1_{\bs_H(h)})\colon  S_{1_{\bs_H(h)}}\to S_{\Phi(h)}$. This follows from the fact that  $H^{G}({B})\times_{\bs_H,\bt}G$ is isomorphic to the holonomy groupoid $H(\overset{\rightarrow}{B})$ (by Thm.~\ref{thm:MoeMrcreg}), and from  Remark~\ref{rem:holbi} about the characterization of holonomy transformations in terms of bisections.

Eq.\eqref{eq:defchi} shows that under the canonical identification  $\Psi_{\bt_H(h)}$, composing  $L_{\alpha}$ with $R_{\Phi(h)^{-1}}$
yields $\chi(h)$.
On the other hand, 
the composition  $R_{\Phi(h)^{-1}}\circ L_{\alpha}$ is exactly $\chi^{conj}(h)$, by definition.
\end{proof}

\begin{example}[Lie subalgebras]\label{ex:liealdescr}
Let  $\g$ be a Lie algebra, $G$ an integrating connected Lie group, and $\h$ a Lie subalgebra. Then $H^G(\h)=:H$ is the connected Lie subgroup of $G$ with Lie algebra $\h$. Fix a slice $S_e\subset G$ through $e$ transverse to $H$.

Now take $h\in H$.  As $M$ is a point, the bisections $\alpha_H$ and $\alpha$ coincide with $h$. Thus $$\chi^{conj}(h)\colon S_e\to h\cdot S_e \cdot h^{-1},\;\; g\mapsto h\cdot g \cdot h^{-1}$$
is conjugation by $h$, restricted to $S_e$.
Notice that in general there is no\footnote{Indeed, infinitesimally this conditions corresponds to the existence of a  subspace of $\g$ which is transverse to $\h$ and which is preserved by $[\h,\cdot]$.} slice $S_e$ that is invariant under conjugation by all elements of $H$. 

A side-remark is that with this at hand, we can describe concretely the map $\sigma|_{S_e}\colon S_e\to H$ used to construct $\chi$  in Ex.~\ref{ex:lieal}. We know that $h\cdot S_e \cdot h^{-1}$ is a slice through $e$ transverse to $H$,
and therefore transverse to the right-translates of $H$ (which are the leaves of $\overset{\rightarrow}{\h}$).
Hence for every $g\in S_e$ there a unique element $\varepsilon(g)\in H$ ``close'' to $e$ such that $\varepsilon(g)\cdot  h g  h^{-1}\in S_e$. We have $\sigma(g)=\varepsilon(g) h$.
 \end{example}

\begin{remark}\label{rem:normal}
Let $\g$ be a Lie algebra, $\h$ a Lie subalgebra, and assume the notation of Ex.~\ref{ex:lieal}.  
We show that $\chi\colon H\to {GermDiff\,}(S_e,S_e)$ is trivial if and only if $H$ is a normal subgroup of $G$.

We do so using Ex.~\ref{ex:liealdescr}.
For all $h\in H$ and $g\in S_e$,  the following are equivalent:

- the elements $g$ and $(\chi^{conj}(h))(g)=h\cdot g \cdot h^{-1}$ lie in the same leaf of $\overset{\rightarrow}{\h}$, 

- the element $(h\cdot g \cdot h^{-1})\cdot g^{-1}$ lies in 
 $H$,
 
 - the element   $g \cdot h^{-1}\cdot g^{-1}$ lies in $H$.
 
 Hence, using Prop.~\ref{prop:agree}, we see that $\chi(h)=Id_{S_e}$ for all $h\in H$ if{f} $H$ is invariant under conjugation by all elements of $S_e$. In that case $\h$ is invariant under $Ad_g$  for all $g\in S_e$, thus 
 $\h$  is a Lie ideal in $\g$, and thus   $H$ is a normal subgroup of $G$.
 \end{remark}


\section{Relation to the Bott connection}\label{sec:Bott}

Let $B$ be a wide Lie subalgebroid of the Lie algebroid $A$ over $M$. There is a flat $B$-connection on $A/B$, defined by $$\nabla_b \underline{a}=\underline{[b,a]},$$ where $b\in \Gamma(B)$, $a\in \Gamma(A)$ and $\underline{a}:=a \text{ mod } B$ \cite[Example 4]{VanEst}.
It generalizes the well-known Bott connection   (the case $A=TM$),
and it plays an important role in \cite{AtHomLeibniz} where the Atiyah class for the pair $(A,B)$ is explored.
The map $\nabla$ is a \emph{Lie algebroid representation} of  $B$ on the vector bundle $A/B$. In other words, it is   a Lie algebroid morphism   $B \to Der(A/B)$
into  the Lie algebroid whose sections are covariant differential operators on the vector bundle $A/B$.
One of its applications is that the 1-cocycles of the Lie algebroid $B$ with values in this representation are exactly the first order deformations of $B$ to Lie subalgebroids of $A$ (this generalizes an observation by Heitsch for foliations \cite{Heitsch}).

In this section we show that the maps $\chi$ and $\chi^{conj}$ introduced earlier, upon a suitable differentiation, yield $\nabla$.
As earlier, let $G\rightrightarrows M$ be a Lie groupoid integrating $A$,  and recall that we have a canonical Lie groupoid morphism $\Phi\colon H^G(B)\to G$. 
Choose smoothly, for every $x\in M$, a slice $S_x$ inside the source fiber $\bs^{-1}(x)$ 
transverse to $B_x$.

From the \emph{Lie groupoid action} $\chi$ (see Thm.~\ref{thm:main}), by differentiating germs of diffeomorphisms, we obtain a \emph{Lie groupoid representation}
$\Psi\colon H^G(B)\to \textrm{Iso}(A/B)$, where the latter is the Lie groupoid  consisting of isomorphisms between the fibers of the vector bundle $A/B\to M$.
The representation is  given by 
$$\Psi(h):=(\chi(h))_*\colon (A/B)_{\bs_H(h)}\to (A/B)_{\bt_H(h)},$$
where we use the canonical identification $T_xS_x\cong (A/B)_x$.
Since $\chi$ and  $\chi^{conj}$ agree upon a canonical identification between slices, we also have $(\chi^{conj}(h))_*=(\chi(h))_*=\Psi(h)$ for all $h$.

The main result of this section is the following proposition, which we  prove in two different ways.
\begin{proposition}\label{prop:Bott}
  The Lie groupoid representation $\Psi\colon H^G(B)\to Iso(A/B)$ differentiates to the Lie algebroid representation $\nabla\colon B\to Der(A/B)$.
\end{proposition}

\begin{example}[Lie subalgebras]\label{ex:liealBott}
Let  $A=\g$ be a Lie algebra, $G$ an integrating connected Lie group,  $B=\h$ a Lie subalgebra, and
$H$ the connected Lie subgroup of $G$ with Lie algebra $\h$. By definition, $\nabla \colon \h\to End(
{\g}/{\h})$ is induced by the Lie bracket of $\g$.
Since  $\chi^{conj}$ is given by conjugation on a slice $S_e\subset G$ transverse to $H$
(see Ex.~\ref{ex:liealdescr}), Prop.~\ref{prop:Bott} implies that  $\Psi$ is the 
 representation of $H$ on ${\g}/{\h}$ induced by the adjoint representation $Ad$. (Recall that $Ad$ is a representation of $G$  on $\g$, and its restriction to $H$ descends to the quotient space ${\g}/{\h}$).
 
Summarizing this analogously to how we did for foliations in the introduction, we have:
\begin{enumerate}
\item  [i)] A Lie group action $\chi$ of $H$ on $S_e$ (induced by conjugation upon an identification of slices),
\item [ii)] A Lie group representation   of $H$ on ${\g}/{\h}$ (induced by $Ad$),
\item [iii)] A Lie algebra representation of $\h$ on ${\g}/{\h}$ (induced by $ad$).
\end{enumerate} 
 
\end{example}
  
\subsection{A proof {of Proposition~\ref{prop:Bott}} in terms of $\chi^{conj}$}\label{subsec:firstproof}

Recall from \S~\ref{sec:conj} that
$\chi^{conj}(h)=R_{\alpha^{-1}}\circ L_{\alpha}$, with the notation introduced there.

Recall also the following fact, which is essentially the content of \cite[Prop. 3.7.1]{MK2}.
\begin{remark}\label{rem:wellknown}
Let $a$ be a section of the Lie algebroid $A$, denote by $\alpha$ the global bisection of $G$ obtained exponentiating $a$ (i.e. the image of $1_M$ under the time-1 flow of the vector field $\rar{a}$; {the time-1 flow is well-defined if, for instance, the image of $a$ under the anchor map is a complete vector field on $M$ \cite[Thm. 3.6.4]{MK2}}). Then the conjugation $R_{\alpha^{-1}}\circ L_{\alpha}$ is a Lie groupoid  automorphism of $G$, whose associated Lie algebroid automorphism is $e^{ad_a}$ (defined as the Lie algebroid automorphism integrating the Lie algebroid derivation $ad_a:=[a,\cdot]$)
\end{remark}

\begin{proof}[First proof of Proposition~\ref{prop:Bott}]
Take a section $b$ of $B$, and for all $\epsilon\in [0,1]$ consider the bisection 
$\alpha_H^{\epsilon}$ of $H^G(B)\rightrightarrows M$ obtained exponentiating the   section $\epsilon b$ of $B$. We have $\frac{d}{d\epsilon}|_0 \alpha_H^{\epsilon}=b$, where to take the derivative  we view bisections as sections of the source map. Denote $\alpha^{\epsilon}=\Phi(\alpha^{\epsilon}_H)$, which is simply  
  the bisection  of $G$ obtained exponentiating $\epsilon b$ viewed as a section of $A$.
 We have
$$\Psi(\alpha_H^{\epsilon})=(\chi^{conj}(\alpha^{\epsilon}))_*
=e^{ad_{  \epsilon b}}$$
using Remark~\ref{rem:wellknown} in the second equality, and denoting by
the same symbol $e^{ad_{  \epsilon b}}$ the map induced on the quotient $A/B$. Hence
$$\left.\frac{d}{d\epsilon}\right|_0\Psi(\alpha_H^{\epsilon})=\left.\frac{d}{d\epsilon}\right|_0e^{ad_{  \epsilon b}}=\nabla_b.$$
\end{proof}

\subsection{A proof {of Proposition~\ref{prop:Bott}} in terms of $\chi$}

 We now prove Prop.~\ref{prop:Bott} by reducing it to the case  of foliations on the Lie groupoid $G$.

\begin{proof}[Second proof of Proposition~\ref{prop:Bott}]
On $G$ we have a right-invariant foliation $\rar{B}$.
Recall from the proof of Thm.~\ref{thm:main}.
that we denoted by $\Xi$ the  holonomy map of (the restriction to $\bs$-fibers) of this foliation.
Taking derivatives of germs of diffeomorphisms,  one obtains the Lie groupoid representation
$$\widehat{\Psi}\colon H(\rar{B})\to Iso (\rar{A}/\rar{B}),\;\; \Psi(\xi)=(\Xi(\xi))_*$$
on the normal bundle of the foliation. By \cite[Lemma 3.11]{AZ2},
 its corresponding Lie algebroid representation is 
given by the Bott connection  of the foliation, which we denote by $\widehat{\nabla}$.

Consider the bundle map  $\ker(\bs_*)=\rar{A}\to A$ covering $\bt$ defined by  $v\in \ker(\bs_*)_g\mapsto (R_{g^{-1}})_*v$, and the induced bundle map $R_*\colon \rar{A}/\rar{B}\to A/B$.
The following two claims relate the Bott connection on $G$ with the  one on $M$, and similarly for the Lie groupoid representations.

\emph{Claim 1: For all sections $b$ of $B$ and $a$ of $A$:
$$\nabla_b(a \text{ mod } B)=R_*\left(\widehat{\nabla}_{\rar{b}}(\rar{a} \text{ mod } \rar{B})\right).$$
}

\emph{
Claim 2: For all $h\in H^G(B)$ we have
$$\Psi(h)=R_*\circ \widehat{\Psi}(\xi)$$
where $\xi\in H(\rar{B})$ is the  unique preimage of $h$ under $\pi$ with source equal to $1_{\bs_H(h)}$.}

We will not prove the claims. We just mention that Claim 1 -- in particular the fact that the section of $\rar{A}/\rar{B}$ appearing there is $R$-projectable --
follows from the definition of the Lie bracket on $\Gamma(A)$ in terms of that of right-invariant vector fields on $G$. Also, in Claim 2 the expression $\Psi(h):=(\chi(h))_*$ is computed directly using eq.~\eqref{eq:defchi}.

Now let $\alpha_H^{\epsilon}$ be a family of sections of the source map of $H^G(B)$ going through the unit section at $\epsilon=0$. Denote by $\xi^{\epsilon}$ the bisection of $H(\rar{B})$ obtained as the  preimage of 
 $\alpha_H^{\epsilon}$ under $\pi$, as in the proof of Prop.~\ref{prop:agree}.
The bisection $\xi^{\epsilon}$ is invariant under the $G$-action by right translations, and $\left.\frac{d}{d\epsilon}\right|_0  \xi^{\epsilon} \in \Gamma(\rar{B})$ is the right-invariant extension of $\left.\frac{d}{d\epsilon}\right|_0  \alpha_H^{\epsilon} \in \Gamma(B)$. Let $a\in \Gamma(A)$. Then
 \begin{align*}
  \left.\frac{d}{d\epsilon}\right|_0(\Psi(\alpha_H^{\epsilon}))\underline{a}
  =\left.\frac{d}{d\epsilon}\right|_0 (R_*\circ \widehat{\Psi}(\xi^{\epsilon}))\rar{(\underline{a})}
  &= R_*\left(( \left.\frac{d}{d\epsilon}\right|_0 \widehat{\Psi}(\xi^{\epsilon})\rar{(\underline{a})}\right)\\
  &=R_* \left({\widehat{\nabla}}_{\left.\frac{d}{d\epsilon}\right|_0  \xi^{\epsilon}}\rar{(\underline{a})}\right)\\
  &=  \nabla_{\left.\frac{d}{d\epsilon}\right|_0  \alpha_H^{\epsilon}}\underline{a}.
\end{align*}
Here in the first equality we used Claim 2 together with eq.~\eqref{eq:rightinv}, in the second the statement from
 \cite[Lemma 3.11]{AZ2}  recalled above, 
 and in the last equality we used Claim 1.
\end{proof}

    \section{The holonomy map $\chi$ for singular subalgebroids}\label{sec:sing}
   
The purpose of this section is to generalize Thm. ~\ref{thm:main} to  singular subalgebroids.  
Let $A$ be an integrable Lie algebroid over $M$. We recall from \cite{AZ3}:
\begin{definition}
  A {\bf singular subalgebroid} of $A$  is a $C^{\infty}(M)$-submodule $\cB$ of 
    $\Gamma_c(A)$  which is locally finitely generated and involutive with respect to the Lie algebroid bracket. Here  $\Gamma_c(A)$ denotes the compactly supported sections of $A$.
\end{definition}
 The two basic examples of singular subalgebroids are the following.
\begin{itemize}
\item[a)] Let $B$ be a wide Lie subalgebroid of $A$. Then $\Gamma_c(B)$
is a singular subalgebroid of $A$.
\item[b)] The singular subalgebroids of the tangent bundle $TM$ are exactly the singular foliations as defined in \cite{AndrSk}.
\end{itemize}

Now fix a source connected Lie groupoid $G\rightrightarrows  M$ integrating $A$.
The holonomy groupoid $H^{G}({\cB})\rightrightarrows M$  of a singular subalgebroid was constructed in \cite[\S 3.2]{AZ3}. For wide Lie subalgebroids it yields the minimal integrals over $G$ introduced by Moerdijk-Mr{\v{c}}un \cite{MMRC}, and its construction is based on   the one of Androulidakis-Skandalis for singular foliations \cite{AndrSk}.
   
For singular subalgebroids we have the following result, which for wide Lie subalgebroids reduces to Thm. ~\ref{thm:main}, and for singular foliations to \cite[Thm. 2.7]{AZ2}.

\begin{theorem}\label{thm:mainsing}
Let $A$ be an integrable Lie algebroid over $M$ and $G$ an integrating source connected Lie groupoid. Let $\cB$ be a  singular subalgebroid of $A$.
Choose\footnote{Since the dimension of the slices is not constant in general, this choice can not be done smoothly.}, for every $x\in M$, a slice $S_x$ through $x$ inside the source fiber $\bs^{-1}(x)$ so that $T_xS_x$ is transverse to ${B}_x:=\{\alpha_x:\alpha\in \cB\}$.  

There is a canonical  groupoid morphism  
$$\chi\colon H^G(\cB)\to \coprod_{{x,y}\in M} \frac{\textit{GermDiff\,}(S_x,S_y)}{exp(I_x \overset{\rightarrow}{\cB}_{S_x})}$$
 mapping $h\in  H^G(\cB)$ to an equivalence class of germs of diffeomorphisms
$$S_{\bs_H(h)}\to S_{ \bt_H(h)}.$$
Above, $I_x$ denotes the functions on $G$ vanishing at $x$, and $\overset{\rightarrow}{\cB}$ is the singular foliation on $G$  generated by right-translations of elements of $\cB$.  By $\overset{\rightarrow}{\cB}_{S_x}$ we denote the restriction of $\overset{\rightarrow}{\cB}$ to $S_x$, and
 $exp(I_x \overset{\rightarrow}{\cB}_{S_x})$
is the space of germs at $x$ of time-1 flows of time-dependent vector fields in the ideal 
$I_x \overset{\rightarrow}{\cB}_{S_x}$.
\end{theorem}   
 
Our motivation for such a result was outlined in the Introduction.
The following proof 
parallels that of  Thm. ~\ref{thm:main}.   
   
 \begin{proof}
We use the alternative construction of $H^{G}({\cB})\rightrightarrows M$ given in \cite[\S 2.1]{AZ4}, which is as follows:
consider the singular foliation $\overset{\rightarrow}{\cB}$ on $G$. The Lie groupoid $G$ acts on $H(\overset{\rightarrow}{\cB})\rightrightarrows G$, the holonomy groupoid of the singular foliation, and the quotient is exactly $H^{G}({\cB})\rightrightarrows M$. 

{\it Step 1.}
By \cite[Thm. 2.7]{AZ2}, the holonomy groupoid of any singular foliation has a canonical effective ``action'' on the slices transverse to the leaves. It is not an honest action: to an element of the holonomy groupoid it does not
associate a germ of diffeomorphism between the corresponding slices, but just 
an equivalence class thereof. {(It is however independent of the choice of slices, up to isomorphism, by \cite[Lemma A.9]{AZ2}.)}
Applying this to the singular foliation $\overset{\rightarrow}{\cB}$ on $G$, and restricting 
 to the slices $S_g:=(R_g)_*S_{\bt(g)}$ inside the  source fibers, yields an 
 injective groupoid morphism 
 $$\Xi\colon H(\overset{\rightarrow}{\cB}) \to  \coprod_{{g,g'}\in G} \frac{\textit{GermDiff\,}(S_g,S_{g'})}{exp(I_g \overset{\rightarrow}{\cB}_{S_g})}.$$

For all elements $\xi \in  H(\overset{\rightarrow}{\cB}) $, the equivalence classes $\Xi(\xi)$ and $\Xi(\xi\cdot g)$ are related as in~\eqref{eq:rightinv}.
To see this, use that  a representative for $\Xi(\xi)$ is given by 
(a suitable restriction of) the diffeomorphism carried by  a smooth bisection\footnote{Here ``smooth'' is mean w.r.t. the diffeological structure on $H(\overset{\rightarrow}{\cB})$ \cite{AZ4}.}
 of $H(\overset{\rightarrow}{\cB})$ through $g$,
and the fact that the source and target maps of $H(\overset{\rightarrow}{\cB})\rightrightarrows G$ are equivariant w.r.t. the right $G$-action.

{\it Step 2.} 
For every $h\in H^G(\cB)$, there is a 
distinguished preimage $\xi\in H(\overset{\rightarrow}{\cB})$ under $\pi$, namely the unique preimage with source equal to $1_x$,
where $x:={\bs_H(h)}$. (Under the isomorphism of Thm.~\ref{thm:trafogroidreg}, which holds also in the singular case by  \cite[\S 2.1]{AZ4}, the distinguished preimage is $(h, 1_{x})\in H^{G}({\cB})\times_{\bs_H,\bt}G$.)
The target of $\xi$ is $\Phi(h)$, hence we have $$\Xi(\xi)\in \frac{\textit{GermDiff\,}(S_{1_{x}},S_{ \Phi(h)})}{exp(I_{1_{x}} \overset{\rightarrow}{\cB}_{S_{1_{x}}})}.$$
We define $\chi$ by 
\begin{equation}
\chi(h):=R_{\Phi(h)}^{-1}\circ \Xi(\xi).  
\end{equation}

{\it Step 3.}
The map $\chi$ is a groupoid morphism. This is shown as in the proof of  Thm. ~\ref{thm:main}.   
\end{proof}

We expect the other results of this note to extend to singular subalgebroids too, with suitable modifications. We briefly comment on this.

\begin{remark}
The results of \S~\ref{subsec:constr}  hold for  singular subalgebroids too. Indeed Thm.~\ref{thm:trafogroidreg} is valid in the singular case. The same is true for Rem.~\ref{rem:holbi}, requiring bisections  to be smooth w.r.t. the natural diffeology on the holonomy groupoid \cite{AZ4}, by the very construction of \cite[Thm. 2.7]{AZ2}.
 
As a consequence, the constructions and results of \S~\ref{sec:conj} also hold for singular subalgebroids.

Regarding  \S~\ref{sec:Bott}, in the singular case 
we no longer have a vector bundle $A/B\to M$, hence a proper generalization of 
Prop.~\ref{prop:Bott} can not be formulated in the same terms.
However we have a vector bundle over each embedded leaf $L\subset M$ of $\cB$, namely the quotient vector bundle $A|_L/B_L$. Here $A|_L$ is\footnote{Despite the similar notation, this should not be confused with the Lie algebroid denoted by $A_L$ in \cite[\S 1.1]{AZ2}.} the restriction of $A$ to $L$, and $B_L:=\coprod_{x\in L} \{\alpha_x: \alpha\in \cB\}$ is the subbundle obtained from $\cB$ by evaluation at points of $L$.
(See \cite[\S 3.2 and \S A.3]{AZ2} for the case of singular foliations.)
Denote by $\cB_L$ the transitive Lie algebroid over $L$ whose compactly supported sections are given by $\cB/I_L\cB$.
We expect a version of
Prop.~\ref{prop:Bott}  restricted to the leaf, which yields a Lie algebroid representation of $\cB_L$ on  $A|_L/B_L$, and which for singular foliations reduces to \cite[Prop. 3.12]{AZ2}. The proof presumably would involve arguments from our second proof of Prop.~\ref{prop:Bott}, and from the proofs of  \cite[Prop. 3.1(2) and Prop. 3.12]{AZ2}. 
    \end{remark}
   
\bibliographystyle{habbrv} 

\end{document}